\documentclass[reqno]{amsart}

\usepackage{amsmath,amsfonts,amssymb,graphics,graphicx,latexsym,amscd,subfigure,hyperref}
\usepackage[usenames]{xcolor}
\usepackage[retainorgcmds]{IEEEtrantools}
\usepackage{graphicx}
\usepackage{amssymb}
\usepackage{xypic}
\usepackage{hyperref}
\usepackage{amsmath,amssymb}
\usepackage{dsfont}\let\mathbb\mathds
\usepackage[english]{babel}

\newtheorem{theorem}{Theorem}[section]
\newtheorem{lemma}[theorem]{Lemma}
\newtheorem{fact}{Fact}

\newtheorem{proposition}[theorem]{Proposition}

\theoremstyle{remark}

\theoremstyle{definition}
\newtheorem{definition}[theorem]{Definition}

\DeclareMathOperator{\re}{Re}

\usepackage{marvosym}

\def\del{\partial}              

\def\bC{\mathbb C}          
\def\bR{\mathbb R}          
\def\bZ{\mathbb Z}          
\def\bT{\mathbb T}          

\def\mC{\mathcal{C}}
\def\mL{\mathcal{L}}            

\def\mV{\mathcal{V}}

\def\mK{\mathcal{K}}

\def\sup{\mbox{Sup}}
\def\spot{\mbox{Spot}}

\def\bi{{\bf{ i}}}
\def\ps{s}
\def\pol{P}
\def\ra{\rightarrow}

\def\vol{d\varpi}

\def\zeta{{\rm A}}
\def\del{\partial}

\def\Lap{\Delta}

\def\hess{{\rm{ Hess\,}} }
\def\tr{{\rm{ Tr \,}} }

\def\lT{\lambda_1^{\bT}}

\newcommand{\newb}[1]{{\color{blue}#1}}

 \newcommand*{\quot}[2]%
{\ensuremath{%
   \raisebox{.35ex}{\ensuremath{#1}}\big/\raisebox{-.35ex}{\ensuremath{#2}}}}

\begin{document}
\title[]{Critical K\"ahler toric metrics for the invariant first eigenvalue}

\author{Rosa Sena-Dias}
\address{Rosa Sena-Dias\\Centro de An\'alise Matem\'atica, Geometria e Sistemas Din\^amicos\\ Departamento de Matem\'{a}tica, Instituto Superior T\'{e}cnico\\ Av. Rovisco Pais, 1049-001 Lisboa\\ Portugal}
\email{rsenadias@math.ist.utl.pt}
\thanks{This work was partially supported by FCT/Portugal through project PTDC/MAT-GEO/1608/2014}
\date{}

\begin{abstract}
In \cite{ls} it is shown that the first eigenvalue of the Laplacian restricted to the space of invariant functions on a toric K\"ahler manifold (i.e. $\lT$, the invariant first eigenvalue) is an unbounded function of the toric K\"ahler metric. In this note we show that, seen as a function on the space of toric K\"ahler metrics on a fixed toric manifold, $\lT$ admits no analytic critical points. We also show that on $S^2$, the first eigenvalue of the Laplacian restricted to the space of $S^1$-equivariant functions of any given integer weight admits no critical points.
\end{abstract}
\maketitle
\section{Introduction}
Let $(M,g)$ be a Riemannian manifold and let $\lambda_1$ denote the first eigenvalue of the Beltrami-Laplace operator on $M$. If we assume that $M$ is of dimension $2$ and has volume $1$ it is well known by a theorem of Yang-Yau that  $\lambda_1$ is a bounded function of the metric $g$ on $M$. One can ask if there is a Riemannian metric which achieves
$$
\sup\{\lambda_1(g)|\, g \, \text{is a Riemannian metric,} \, \text{vol}(g)=1\}.
$$
For $S^2$, this metric is known to be the Fubini-Study metric. In \cite{n}, Nadirashvili studies the same problem for $\bT^2$. He defines the notion of $\lambda_1$-critical metric which is roughly speaking a critical point for the function $\lambda_1(g)$. Note that $\lambda_1$ is not a differentiable function of $g$ in general so this definition requires some care. We will say more on this ahead. In higher dimensional Riemannian manifolds El Soufi-Ilias, generalising a result of Nadirashvili, prove the following characterisation of $\lambda_1$-critical metrics
\begin{theorem}[El Soufi-Ilias, Nadirashvili] A Riemannian metric $g$ on $M$ is critical for $\lambda_1$ iff $g$ admits a set of eigenfunction $\{f_a, a=0,\cdots, N\}$ for $\lambda_1(g)$ such that $F=(f_0,\cdots,f_N)$ embeds $M$ into $S^N$, with $g=F^*g_{FS}$ and $F(M)$ minimal in $S^N$.
\end{theorem}
Therefore $\lambda_1$-critical metrics yield minimal  submanifolds of spheres.

We are interested in the more symmetric case when $(M,g)$ admits an isometric group action by a group  $G$. In \cite{cde}, Colbois-Dryden-El Soufi introduce the notion of $\lambda_1^G$-critical invariant metrics where $\lambda_1^G$ is the smallest positive eigenvalue of the Laplacian restricted to $G$-invariant eigenfunctions. Again this notion is subtle as $\lambda_1^G$ is not in general a differentiable function the invariant metric but it is analogous to the notion introduced by Nadirashvili. They prove the following theorem
\begin{theorem}[Colbois-Dryden-El Soufi] If $G$ has dimension greater than $1$ then $M$ admits no $G$-invariant metric which is critical for $\lambda_1^G$.
\end{theorem}
Given a group character $\chi$ it is easy to generalize the above notions to the setting of $\chi$-equivariant functions. These are functions $f:M\ra \bC$ that satisfy $f(h\cdot x)=\chi(h)f(x)$, for all $x\in M$, $h\in G$. We have a notion of equivariant first eigenvalue $\lambda_1^\chi$ and $\lambda_1^\chi$-critical metric.

More specifically we are interested in the case of toric manifolds. These are symplectic manifolds $(M^{2n},\omega)$ admitting a Hamiltonian $\bT^n$-action. Symplectic toric manifolds always admit a large family of compatible integrable $\bT^n$-invariant complex structures thus they carry several K\"ahler structures (see \cite{g} , \cite{a}). In fact for a fixed $\omega$, toric K\"ahler structures in the class $[\omega]$ are very well understood and are parametrised by a subset of the set of continuous functions on the moment polytope of $(M,\omega,\bT^n)$ which we denote by $\spot(M,\omega,\bT^n)$ and which we will describe carefully in the next section. We want to think of $\lT$  as a function on $\spot$. That is, we want to consider only toric {\it K\"ahler} metrics in the class $[\omega]$. Because we are not considering all $\bT^n$ invariant functions the results in \cite{cde} do not apply to our setting (except in dimension $2$). 

There has recently been an interest in considering spectral problems in the realm of K\"ahler geometry. In \cite{ajk} the authors define $\lambda_1$-extremal K\"ahler metric on a K\"ahler manifold as being those which are critical for $\lambda_1$ restricted to the space of K\"ahler metrics in a given class.

We will define an analogous notion of criticality in our setting. More specifically given a toric K\"ahler manifold we are looking for torus invariant K\"ahler metrics which are critical for $\lT$. In this note our goal is prove the following theorems
\begin{theorem}\label{alld}
Let $(M,\omega,g,\bT^n)$ be a toric K\"ahler manifold. Then, there are no analytic toric K\"ahler structures compatible with $\omega$ and in the class $[\omega]$ which are critical for $\lT$.
\end{theorem}
Given $k\in \bZ$, $k$ corresponds to an $S^1$-character. We will prove the following
\begin{theorem}\label{s2}
Let $k$ be an integer. There are no $\lambda_1^k$ critical $S^1$-invariant metrics on $S^2.$
\end{theorem}
When $k=1$ this is a consequence of the Colbois-Dryden-El Soufi theorem from above.

We would like be able to remove the analyticity assumption. It is know due to results of Morrey that solutions to elliptic systems of PDE's whose coefficients are analytic have analytic solution if any. We will see that critical toric K\"ahler metrics and their eigenfunctions for the smallest eigenvalue  are solutions to a system of PDE's whose coefficients are analytic. Unfortunately the system is not elliptic. 

This paper is organised in the following way: in section \ref{back} we give some background on $\lambda_1$-critical metrics and on toric K\"ahler geometry, in section \ref{crit} we use the techniques  developed to deal with criticality in the Riemannian case and adapt them to our setting so as to extract a useful characterisation of $\lT$-critical metrics. We then use this characterisation to derive our main theorems in section \ref{proof}. The last section is somewhat independent of the rest of the paper. There, we show that there is an obvious system of PDEs that is satisfied the pair toric K\"ahler metric/corresponding eigenfunctions but the system is nowhere elliptic. 

\noindent \textbf{Acknowledgements.} I would like to thank Christine Breiner and Heather Macbeth for many illuminating conversations about $\lambda_1$-critical metrics.

\section{Background}\label{back}
\subsection{$\lambda_1$-critical  metrics}\label{back_crit}
Let $(M,g)$ be a Riemannian manifold. To fix conventions our Laplacian is given by $\Lap=d^*d$ and has positive eigenvalues. In coordinates $x_i$ on $M$ write $g=g_{ij}dx_i\otimes dx_j$. The Laplacian of a function $f$ on $M$ is given by
\begin{equation}\label{lap_in_coord}
\Lap f=-\frac{1}{\sqrt{\vol}}\frac{\del}{\del x_i}\left(\sqrt{\vol}g^{ij}\frac{\del f}{\del x_j}\right),
\end{equation}
where $g^{ij}$ denote the entries of the inverse of the matrix $\{g_{ij}\}$ and $\vol=\det{g_{ij}}$.
The smallest eigenvalue of the Laplacian is called first eigenvalue and is denoted by $\lambda_1(M,g)$. If we fix $M$, then $\lambda_1$ can be seen as a function on the space of all Riemannian metrics on $M$. Its is not a differentiable function of $g$ but it is Lipschitz. In fact given a one-parameter family of Riemannian metrics on $M$, $g_t$ with $g_0=g$  and analytic in $t$,  if $\lambda_1(g)$ is a multiple eigenvalue, then $\lambda_1$ may become non-differentiable at $g$.  Despite this, there are real valued functions $\Lambda_{0,t}, \cdots \Lambda_{N,t}$ and one parameters families of functions on $M$ $f_{0,t}, \cdots f_{N,t}$ satisfying
$$
\Lap f_{l,t}=  \Lambda_{l,t} f_{l,t}, \quad l=0,\cdots N
$$
and such that $\lambda_1(g_t)=\min\{\Lambda_{l,t},\, l=1,\cdots N\}$ so that the function  $\lambda_1(g_t)$ has a right and left derivative
$$
\frac{d\lambda_1(g_t)}{dt}(0^+)=\min  \left\{\frac{d\Lambda_{l,t}} {dt}(0), \, l=0,\cdots N \right\}
$$
$$
\frac{d\lambda_1(g_t)}{dt}(0^-)=\max \left\{\frac{d\Lambda_{l,t}} {dt}(0), \, l=0,\cdots  N \right\}
$$
\begin{definition} The metric $g$ is $\lambda_1$-critical if for any $1$-parameter family of metrics $g_t$ analytic in $t$ 
$$
\frac{d\lambda_1(g_t)}{dt}(0^-)\cdot \frac{d\lambda_1(g_t)}{dt}(0^+)<0.
$$
\end{definition}
(see \cite{n} and \cite{ei} for more details).
\subsection{Toric Geometry}
We will try to be brief and assume some familiarity with the subject. For more details see \cite{g} and \cite{a}. 
\begin{definition}
A K\"ahler manifold $(M,\omega,g)$ where $\omega$ is a symplectic form and $g$ is a Riemannian metric is said to be toric if it admits an isometric, Hamiltonian $\bT^n$-action. 
\end{definition}
In this case there is a moment map associated to the action $\phi:M\ra (\mbox{Lie}(\bT^n))^*\simeq \bR^n$ and the moment map image $\pol$ is a convex polytope of a special type (a Delzant polytope). In particular it can be written in the form
$$
\pol=\left\{x\in \bR^n: x\cdot \nu_k-c_k>0, \, k=1,\cdots, d\right\}
$$
and at every vertex, there is an $SL(n,\bZ)$ transformation taking a neighbourhood of that vertex into a neighbourhood of $0$ in
$$
\left\{x\in \bR^n: x_k>0, \, k=1,\cdots, n\right\}.
$$
There is an open dense set in $M$ which we denote by $M^0$ where $\bT^n$ acts freely and there is an equivariant symplectomorphism $\psi: M^0\ra \pol\times \bR^n$ whose first factor is given by the moment map $\phi$. Here the $\bT^n$-action on $ \pol\times \bR^n$ is given by the usual $\bT^n$-action on the second factor. Said differently, there are $\bT^n$-equivariant Darboux coordinates $(x,\theta)$ on $M^0$. We refer to these as action-angle coordinates. 

Given a polytope in $\bR^n$ of Delzant type one can construct from it a toric K\"ahler manifold $M_\pol$ in a canonical manner (see \cite{g}). It was shown by Delzant that in fact $\pol$ determines $(M,\omega)$ up to symplectomorphism. Abreu showed there is an effective way to parametrize all compatible $\bT^n$-invariant K\"ahler metrics. 

\begin{definition}\label{spot}
Let $\pol$ be a Delzant polytope. A function $\ps\in \mathcal{C}^\infty(\pol)$ is called a symplectic potential if
\begin{itemize}
\item $\hess \ps$ is positive definite,
\item $\ps-\sum_{k=1}^d \left(x\cdot \nu_k-c_k\right)\log(x\cdot \nu_k-c_k)$ is smooth on $\bar\pol$,
\item $\hess \ps$ when restricted to each face of $\pol$ is positive definite.
\end{itemize}
We denote the set of all such functions by $\spot(\pol)$.
\end{definition}
One can associate to each $\ps\in \spot(\pol)$ a K\"ahler structure $g_\ps$ whose corresponding K\"ahler metric in action-angle coordinates can be written as
$$
(\ps)_{ij}dx_i\otimes dx_j+(\ps)^{ij}d\theta_i\otimes d\theta_j.
$$
In fact it can be shown that all toric K\"ahler structures arise this way. 
The K\"ahler structure constructed in \cite{g} is called the Guillemin K\"ahler structure. Its symplectic potential is
$$
\ps_G=\sum_{k=1}^d \left(x\cdot \nu_k-c_k\right)\log(x\cdot \nu_k-c_k)-\left(x\cdot \nu_k-c_k\right).
$$

We make use of the following very elementary fact:
\begin{fact}
Smooth $\bT^n$-invariant functions on a toric K\"ahler manifold $M$ are in 1 to 1 correspondence with smooth functions on the closure of the moment polytope, $\bar{\pol}$ of $M$.
\end{fact}
\begin{proof}
We denote the space of smooth $\bT^n$-invariant functions by $\mC^\infty_T(M)$. Denote the moment map for the $\bT^n$-action by $\phi$. Given an invariant function $F$ on $M$, set $f$ to be $f(x)=F(\phi^{-1}(x))$. This is well defined because $\phi(p)=\phi(q)$ implies $p$ and $q$ are in the same $\bT^n$-orbit and $F$ is invariant. Conversely given $f\in \mC^\infty(\pol)$, we define $F=f\circ \phi$.
\end{proof}
Similarly we have:
\begin{fact}
Continuous $\bT^n$-equivariant complex functions on a toric K\"ahler manifold $M$ are in 1 to 1 correspondence with continuous complex functions on the closure of the moment polytope $\bar{\pol}$ of $M$ that vanish on $\partial\pol$
\end{fact}
\begin{proof}
Characters in $\bT^n$ can be identified with elements in $\bZ^n$. Given $k\in\bZ^n$ we denote the space of continuous $k$-equivariant functions by $\mC_k(M)$. 

We start by noting that if $F:M\ra\bC$ is $k$-equivariant for $k\ne 0$ then $F$ vanishes on points with non trivial isotropy. Let $F$ be equivariant. If $p$ is a point where $\bT^n$ does not act freely i.e. if $\phi(p)\in \partial \pol$ then for $e^{\bi\theta}$ non-trivial in the stabiliser group of $p,$ $F(e^{\bi\theta}p)=F(p)=e^{\bi\theta\cdot }F(p)$ so that $F(p)=0.$ 

Let $\psi: M^0\ra \pol\times \bR^n$ denote the action-angle coordinates map. If $f$ is a function on $\pol$, we define a $k$-equivariant function on $M^0$ by setting $F\circ\psi^{-1}(x,\theta)=f(x)e^{\bi k\cdot \theta}$. If $f$ vanishes on $\partial\pol$ we can extend $F$ by continuity to $M$ to be zero on $M\setminus M^0$. Conversely, given $F$ $k$-equivariant, define $f$ on $\pol$ by $f(x)=F\circ\psi^{-1}(x,0)$ and extend by $0$ to the boundary. As we have seen $F$ vanishes on $M \setminus M^0$ and $\phi(M\setminus M^0)=\partial \pol$ so that $f$ is continuous on $\partial \pol$.
\end{proof}
\subsection{Equivariant spectrum on toric manifolds}
Let $(M,g)$ be a Riemannian manifold with an isometric $G$-action. Let $\chi$ be a group character and let $\mC^\chi(M)$ denote the set of continuous $\chi$-equivariant  functions.
$$
\mC^\chi(M)=\{F\in \mC(M,\bC):F(h\cdot p)=\chi(h)F(p),\,\forall h\in G\}.
$$
The Laplacian induced from $g$ commutes with the $G$-action because $G$ acts by isometries hence it restricts to $\mC^\chi(M)\cap \mC^\infty(M)$ for any given character of the group $G$. 
\begin{definition}
Let $(M,g)$ be a Riemannian manifold with an isometric $G$-action. The $\chi$-equivariant first eigenvalue is the smallest eigenvalue of $\Lap_{|\mC^\chi(M)\cap \mC^\infty(M)}$ i.e.
$$
\lambda_1^\chi(M,g, G)=\sup\left\{\frac{\int_M|dF|^2\vol_g}{\int_M|F|^2\vol_g}, \, F\in \mC^\chi(M)\cap \mC^\infty(M)\right\}.
$$
\end{definition}
When $\chi$ is the trivial character we often write $\lambda_1^\chi=\lambda_1^G.$ 

We will be using these notions in the setting of Toric K\"ahler manifolds and we will think of $\lambda_1^k$ as a function of the symplectic potential inducing the K\"ahler metric i.e. given $(M,\omega,\bT^n)$ symplectic toric with moment polytope $\pol$ and given $k\in \bZ^n$, we consider
$$
\lambda_1^k:\spot(\pol)\ra \bR^+
$$
and its variations.
 
Given $k\in \bZ^n$, if $F$ is $k$-equivariant, it can be written in action-angle coordinates as $f(x)e^{\bi k\cdot \theta}$ so that from equation (\ref{lap_in_coord}) we have 
\begin{equation}
\Lap F=-e^{\bi k\cdot \theta}\left(\frac{\del}{\del x_i}\left(\ps^{ij}\frac{\del f}{\del x_j}\right)-f k_ik_j\ps_{ij}.\right)
\end{equation}
Note that because $(x,\theta)$ are Darboux coordinates $\vol=1$.
The space of $k$-equivariant eigenfunctions for $\lambda_1^k$, which we denote by $E_1^k$, (or $E_1^\bT$ if $k=0$ in the invariant case) can be identified with a subset of $\mC^\infty(\pol)$. Namely, if $k\ne 0$
$$
E_1^k\simeq\left\{f\in \mC^\infty(\pol):\frac{\del}{\del x_i}\left(\ps^{ij}\frac{\del f}{\del x_j}\right)-f k^t\hess (\ps)k=\newb{-}\lambda_1^k f, \, f=0 \,\text{in} \,\del \pol\right\}
$$
and 
$$
E_1^\bT\simeq\left\{f\in \mC^\infty(\pol):\frac{\del}{\del x_i}\left(\ps^{ij}\frac{\del f}{\del x_j}\right)=\newb{-}\lT f\right\}.
$$
In the invariant case, we often identify $f\in\mC(\pol)$ with the associated eigenfunction on $M$ i.e. we confuse $f$ with $f\circ\phi$ and we write $\Lap f$ to mean $-\frac{\del}{\del x_i}\left(\ps^{ij}\frac{\del f}{\del x_j}\right).$
\section{Critical $\lT$, $\lambda_1^k$ metrics}\label{crit}
In this section we fix a toric symplectic manifold $(M,\omega,\bT^n)$ with moment polytope $\pol$. The first goal is to define critical metrics for the invariant/equivariant first eigenvalue. This is almost exactly a repetition of subsection \ref{back_crit}. To avoid the repetition and give a more unified treatment of the equivariant extremization problem and the classical extremization problem, we could have used the framework developed by Macbeth in \cite{mb}. This would involve showing that the measure described in the main theorem there is of a special type because the spaces $E_1^k$ and $E_1^\bT$ are finite dimensional. 

Instead we will go through the argument in subsection \ref{back_crit} again. We want to define critical values for $\lambda_1^k:\spot(\pol)\ra \bR^+$ but as in the Riemannian case discussed in subsection \ref{back_crit} $\lambda_1^k:\spot(\pol)\ra \bR^+$ is not a differentiable function at all points.  Given a one parameter family in $\spot(\pol)$ with $\ps_0=\ps$  and analytic in $t$, there are real valued functions $\Lambda_{0,t}, \cdots \Lambda_{N,t}$ and one parameters families of functions on $\pol$, $f_{0,t}, \cdots f_{N,t}$ satisfying
$$
\Lap f_{l,t}+f_{l,t}k^t\hess (\ps) k=  \Lambda_{l,t} f_{l,t}, \quad k=0,\cdots N.
$$
and such that $\lambda_1^k(\ps_t)=\min\{\Lambda_{l,t},\, l=1,\cdots N\}$ so that the function  $\lambda_1^k$ has a right and left derivative
$$
\frac{d\lambda_1^k(\ps_t)}{dt}(0^+)=\min  \left\{\frac{d\Lambda_{l,t}} {dt}(0), \, l=0,\cdots N \right\}
$$
$$
\frac{d\lambda_1^k(\ps_t)}{dt}(0^-)=\max \left\{\frac{d\Lambda_{l,t}} {dt}(0), \, l=0,\cdots  N \right\}
$$
\begin{definition} The symplectic potential $\ps$ is $\lambda_1^k$-critical if for any $1$-parameter family of symplectic potentials $\ps_t$, analytic in $t$, 
$$
\frac{d\lambda_1^k(\ps_t)}{dt}(0^-)\cdot \frac{d\lambda_1^k(\ps_t)}{dt}(0^+)<0.
$$
\end{definition}
Setting $\delta \ps$ to be
$
\frac{d \ps}{dt}(0),
$
we write $d\lambda_1^k(f_l,\delta \ps)=\frac{d\Lambda_{l,t}}{dt}(0).$ In fact, we can define $d\lambda_1^k(f,\delta \ps)$ for any $f\in E_1^k$ as follows. Consider the Riemannian metrics corresponding to  $\ps_t=\ps+t\delta \ps$ for $t$ sufficiently small. For each such $t,$ $E_1^k(\ps_t)$ is the first $k$-equivariant eigenspace. We extend $f$ to a one parameter family $f_t$ such that $f_t\in E_1^k(\ps_t)$ and let $\Lambda_t$ be the eigenvalue corresponding to $f_t$. 
$$
d\lambda_1^k(f,\delta \ps)=\frac{d\Lambda_{t}}{dt}(0).
$$
As we will see ahead this does not actually depend on $f_t$ but only on $f$. In fact the same phenomenon occurs in the non equivariant case of subsection \ref{back_crit} and is a manifestation of something more general that is explained and exploited in \cite{mb}.

We now use the toric framework to calculate $d\lambda_1^k(f,\delta \ps).$
\begin{lemma}
Let $(M,\omega,\bT^n)$ be toric with moment polytope $\pol$. Let $\ps\in \spot(P)$. Given $\delta \ps\in \mC^\infty(\pol)$ such that $\delta \ps$ and $d\delta \ps$ vanish on $\del\pol$ and $f\in \mC^{\infty}(\pol)$ corresponding to an eigenfunction of the Laplacian associated to $\ps$,
$$
d\lT(f,\delta \ps)=-\int_\pol \frac{\del^2 \left(\ps^{il}f_l \ps^{jr}f_r\right)}{\del x_i \del x_j}\delta \ps dx,
$$
where we write $f_r$ for $\frac{\del f}{\del x_r}$. If furthermore $f=0$ on $\del \pol$ then 
$$
d\lambda_1^k(f,\delta \ps)=\int_\pol \left(-\frac{\del^2 \re\left(\ps^{il}f_l \ps^{jr}\bar{f}_r\right)}{\del x_i \del x_j}+k^t\hess |f|^2k\right)\delta \ps dx.
$$
\end{lemma}
\begin{proof}
Consider the path $\ps_t=\ps+t\delta \ps$ in $\spot(\pol),$ the corresponding path of Riemannian metrics on $M$ which we denote by $g_t$ and a path $f_t$ in $\mC(\pol)$ corresponding to a path of eigenfunctions in $E_1^\bT(g_{t}),$ the eigenspace for the smallest invariant eigenfunction for the Laplacian associated with $g_t,$ such that $f_0=f.$ We have $\Lap_t f_t={\lT}_t f_t.$ We want to calculate 
$$
\frac{d}{dt}_{|t=0}\lT(f_t,g_t).
$$
We may assume that $\int_\pol f_t^2dx=1$ for all $t$ and taking derivatives this implies $\int_\pol f\dot{f}dx=0$ where $\dot{f}=\frac{d f_t}{dt}.$ The quantity $ \frac{d}{dt}_{|t=0}\lT(f_t,g_t)$ is given by
\begin{IEEEeqnarray*}{l}
\frac{d}{dt}_{|t=0}\int_\pol |df_t|^2_{g_t}dx\\
=\frac{d}{dt}_{|t=0}\int_\pol (\del f_t)^t\hess^{-1}(\ps_t)\del f_t dx\\
=-\int_\pol (\del f)^t\hess^{-1}(\ps)\frac{d\hess (\ps_t)}{dt}_{|t=0}\hess^{-1}(\ps)\del f dx+2\int_\pol (\del \dot{f})^t\hess^{-1}(\ps)\del f dx\\
=-\int_\pol (\del f)^t\hess^{-1}(\ps){\hess (\delta \ps)}\hess^{-1}(\ps)\del f_t dx+2\int_M \langle d\dot (f\circ \phi),d(f\circ\phi)\rangle dx\\
=-\int_\pol (\del f)^t\hess^{-1}(\ps){\hess (\delta \ps)}\hess^{-1}(\ps)\del f_t dx+2\int_M \dot f\circ \phi \Lap (f\circ\phi) dx\\
=-\int_\pol (\del f)^t\hess^{-1}(\ps){\hess (\delta \ps)}\hess^{-1}(\ps)\del f_t dx+2\lT\int_\pol \dot f  f dx\\
=-\int_\pol (\del f)^t\hess^{-1}(\ps){\hess (\delta \ps)}\hess^{-1}(\ps)\del f_t dx\\
=- \int_\pol\left(\ps^{il}f_l \ps^{jr}f_r\right)(\delta \ps)_{ij}dx,\\
\end{IEEEeqnarray*}
where we have used $\phi$ to mean the moment map for the torus action on $M$. The conditions that $\ps$ and $d\ps$ vanish on $\partial \pol$ ensure that we can integrate the above by parts without picking up boundary terms and hence
$$
\frac{d}{dt}_{|t=0}\lT(f_t,g_t)=- \int_\pol\frac{\del^2\left(\ps^{il}f_l \ps^{jr}f_r\right)}{\del x_i\del x_j}(\delta \ps) dx,
$$
as claimed. The $k$-equivariant case is similar.
\begin{IEEEeqnarray*}{l}
\frac{d}{dt}_{|t=0}\lambda_1^k(f_t,g_t)\\
=\frac{d}{dt}_{|t=0}\int_\pol |d(e^{\bi k\cdot \theta}f_t)|^2_{g_t}dx\\
=\frac{d}{dt}_{|t=0}\int_\pol \left(\re\left( (\del f_t)^t\hess^{-1}(\ps_t)\del \bar{f}_t \right) +|f_t|^2k^t\hess (\ps_t) k\right)dx\\
= \int_\pol\left(-\re(\ps^{il}f_l \ps^{jr}\bar{f}_r)+|f|^2k_ik_j\right)(\delta \ps)_{ij}dx.\\
\end{IEEEeqnarray*}
Integrating by parts we get 
$$
\frac{d}{dt}_{|t=0}\lT(f_t,g_t)= \int_\pol\left(-\frac{\del^2\re\left(\ps^{il}f_l \ps^{jr}\bar{f}_r\right)}{\del x_i\del x_j}+k^t\hess|f|^2 k\right)\delta \ps dx.
$$
\end{proof}
We are now ready to prove our main characterisation of $\lT$-critical metrics in this section.
\begin{proposition}\label{characterisation_crit}
In the same setting as above, the symplectic potential $\ps$ is $\lambda_1^k$-critical iff for all $\delta \ps\in \mC^\infty(\bar{\pol})$ there are functions on $\pol$, $\{f_0,\cdots, f_N\}$, corresponding to $k$-equivariant eigenfunctions in $E_1^k(\ps)$ and $\alpha_0,\cdots, \alpha_N \in [0,1]$ satisfying
$$
\sum_{a=1}^N \alpha_a \left(\left(\frac{\del^2 \re \left(\ps^{il}f_{a,l} \ps^{jr}\bar{f}_{a,r}\right)}{\del x_i \del x_j}\right)-k^t\hess |f_a|^2k\right)=0.
$$
\end{proposition}
Again this lemma has an analogous counterpart in the classical critical first eigenvalue problem and it is a manifestation of a more general phenomenon which is treated in \cite{mb}. To use Macbeth's results in our setting, we would need to prove that the measure described in the main theorem there is of a special type  (the relevant fact being that $E_1^k(\ps)$ is finite dimensional). We have chosen to derive the results so as to be self-contained.
\begin{proof}
The condition that $\ps$ is critical can be rewritten as 
$$
\ps \, \mbox{ is critical} \,  \iff \forall \delta \ps\in \mC^\infty(\bar{\pol}),\, \exists \,f,h\in E_1^k(\ps): d\lambda_1^k(f,\delta \ps)<0<d\lambda_1^k(h,\delta \ps).
$$
Now fix $\delta \ps\in \mC^\infty(\bar{\pol})$ and consider $d\lambda_1^k(.,\delta \ps)$ as a function on the finite dimensional vector space $E_1^k(\ps)$. By restriction to the sphere in $E_1^k(\ps)$ with respect to the $\mL^2$ norm we see that
$$
\ps \, \mbox{ is critical} \,  \implies \forall \delta \ps\in \mC^\infty(\bar{\pol}),\, \exists \,f\in E_1^k(\ps),\,\int_P|f|^2dx=1: d\lambda_1^k(f,\delta \ps)=0.
$$
The relevant thing to note is that multiplying $f$ by a fixed constant changes $d\lambda_1^k(f,\delta \ps)$ by multiplication by a positive constant. Now assume that $\delta \ps$ and its derivatives vanish along $\del \pol$ so that from the previous lemma
$$
d\lambda_1^k(f,\delta \ps)=\int_\pol \left(-\re\left(\frac{\del^2 \left(\ps^{il}f_l \ps^{jr}\bar{f}_r\right)}{\del x_i \del x_j}\right)+k^t\hess |f|^2k\right)\delta \ps dx.
$$
We set 
$$
Q_{\ps}(f)=-\re\left(\frac{\del^2 \left(\ps^{il}f_l \ps^{jr}\bar{f}_r\right)}{\del x_i \del x_j}\right)+k^t\hess |f|^2k,
$$
so that
$$
d\lambda_1^k(f,\delta \ps)=\int_\pol Q_{\ps}(f)\delta \ps dx.
$$
If $\ps$ is critical
$$
\forall \delta \ps\in \mC^\infty(\bar{\pol}),\delta \ps,d(\delta\ps)=0 \,\mbox{on}\, \del \pol,\, \exists \,f\in E_1^k(\ps): \int_\pol |f|^2=1,\, \int_\pol Q_{\ps}(f)\delta \ps=0 .
$$
We want to prove that $0$ is the convex hull generated by $\{Q_\ps(f), f\in E_1^k(\ps)\}$. Let $\mK$ be this convex hull. Suppose $0\notin \mK$. By the Hahn-Banach separation theorem applied in $\mL_k^2(M)$ (the $\mL^2$ completion of the space of $k$-equivariant functions on $M$ which we are identifying with a subspace of $\mL^2(\pol)$) there is $\mu$, a linear bounded functional on $\mL_k^2(M)$ such that $\mu_{|\mK}>0$. By Riesz's representation theorem there is $\beta\in \mL^2(\pol)$ such that 
$$
\mu(h)=\int_\pol\beta hdx>0, \, \forall h\in \mK.
$$
Suppose that $\beta$ and its first order derivatives vanish on $\del \pol$. Then because $\ps$ is critical there is $f\in E_1^k(\ps)$ with $\mL^2$-norm equal to 1 such that
$$
\int_\pol Q_\ps(f)\beta=0,
$$
but by assumption $\int_\pol Q_\ps(f)\beta=\mu(Q_\ps(f))>0$ because $Q_\ps(f)\in \mK$ and we get a contradiction. But since $\beta$ (or its first order derivatives) may not vanish on $\del \pol$, we need a slight modification of the above argument. Consider the smooth bump function $\rho_\epsilon$ which is identically equal to $1$ on $\pol\setminus \mV_\epsilon (\del \pol)$ where $\mV_\epsilon (\del \pol)$ denotes a tubular neighbourhood of radius $\epsilon$ of $\del\pol$. Let $\beta_\epsilon$ denote $\rho_\epsilon\beta$. Then because $\ps$ is critical, there is $f_\epsilon \in E_1^k(\ps)$ with $\mL^2$-norm equal to 1 such that
$$
\int_\pol Q_\ps(f_\epsilon)\beta_\epsilon=0,
$$
and $\int_\pol |f_\epsilon|^2=1.$ Now $\{f_\epsilon\}$ is bounded and contained in a finite dimensional space so that it admits a convergent subsequence. Let $f\in E_1^k(\ps)$ be the limit. Because the subsequence converges in that finite dimensional subspace, $Q_\ps(f_\epsilon)$ converges to $Q_\ps(f)$ in the same subsequence. The sequence $\beta_\epsilon$ also converges a.e. to $\beta$ so that $Q_\ps(f_\epsilon)\beta_\epsilon$ has a subsequence that converges a.e. to $Q_\ps(f)\beta$. On the other hand for that subsequence $|Q_\ps(f_\epsilon)\beta_\epsilon|\leq C \beta$ for some constant $C$. This is because in the subsequence there is bound on the $\mL^\infty$-norm of $Q_\ps(f_\epsilon)$. By the bounded convergence theorem 
$$
\int_\pol Q_\ps(f_\epsilon)\beta_\epsilon\ra \int_\pol Q_\ps(f)\beta=0.
$$
But $\int_\pol Q_\ps(f)\beta=\mu(Q_\ps(f))>0$ because $Q_\ps(f)\in \mK$ and we get a contradiction. We conclude that $0\in \mK$ and the proposition follows.
\end{proof}
\section{Proof the of main Theorems \ref{s2}, \ref{alld}}\label{proof}
The idea is to exploit the characterisation given in Proposition \ref{characterisation_crit} for critical toric K\"ahler metrics to conclude that such metrics do not exist.
\subsection{The proof of theorem \ref{s2}}
\begin{proof}
Let $k\in \bZ$ be fixed. Under the right normalisation, $(S^2,\omega_{FS},S^1)$ is a toric symplectic manifold with moment polytope $]-1,1[$. Any $S^1$-invariant metric on $S^2$ is described by a symplectic potencial $\ps\in \spot(]-1,1[)$. From Proposition \ref{characterisation_crit} if it is critical for $\lambda_1^k$ then there are function $\{f_0,\cdots, f_N\}$ and $\alpha_0,\cdots a_N\in [0,1]$ satisfying
\begin{equation}\label{lap_S2}
\left(\frac{f_a'}{\ps''}\right)'=\left(\newb{-}\lambda+k^2\ps''\right)f_a
\end{equation}
and 
$$
\sum_{a=0}^N\alpha_a\left( \left|\frac{f_a'}{\ps''}\right|^2-k^2|f_a|^2\right)''=0.
$$
As the $\alpha_a$ are all positive (and smaller than $1$) they can be absorbed into the $f_a$'s at the cost of loosing the normalisation for $\int_P |f_a|^2dx$'s. We write 
$$
\sum_{a=0}^N \left( \left|\frac{f_a'}{\ps''}\right|^2-k^2|f_a|^2\right)''=0.
$$ Now
$$
\sum_{a=0}^N\left( \left|\frac{f_a'}{\ps''}\right|^2-k^2|f|^2\right)'=2\re\left( \left(\frac{f_a'}{\ps''}\right)' \frac{\bar{f}_a'}{\ps''}-k^2f_a'\bar{f_a}\right)
$$
and replacing in Equality (\ref{lap_S2}) we see that 
$$
\sum_{a=0}^N\left( \left|\frac{f_a'}{\ps''}\right|^2-k^2|f|^2\right)'=2\sum_{a=0}^N\re\left( \left(\newb{-}\lambda+k^2\ps''\right)f_a \frac{\bar{f}_a'}{\ps''}-k^2f_a'\bar{f_a}\right)
$$
$$
\sum_{a=0}^N\left( \left|\frac{f_a'}{\ps''}\right|^2-k^2|f|^2\right)'=\newb{-}2\lambda \sum_{a=0}^N\frac{\re(f_a\bar{f}_a')}{\ps''}.
$$
This then implies that 
$$
\sum_{a=0}^N\frac{\re(f_a\bar{f}_a')}{\ps''}
$$
is constant. Because $\frac{1}{\ps''}$ vanishes at $1$ and $-1$, this is actually zero and $\sum_{a=0}^N{\re(f_a\bar{f}_a')}=0$ so that $\sum_{a=0}^N|f_a|^2$ is constant. We look at two cases separately:
\begin{itemize}
\item In the case where $k\ne 0$, the $f_a$ all vanish at $1$ and $-1$ and so 
$\sum_{a=0}^N |f_a|^2=0$ so that $f_a=0$ for all $a$; a contradiction. 
\item In the case when $k=0$ we may assume that the $f_a$ are real. We have
$$
\sum_{a=0}^N \left( \left(\frac{f_a'}{\ps''}\right)^2\right)''=2\sum_{a=0}^N \left( \left(\frac{f_a'}{\ps''}\right)''\frac{f_a'}{\ps''}+\left(\left(\frac{f_a'}{\ps''}\right)'\right)^2\right)=0
$$
and replacing Equality (\ref{lap_S2}) for $k=0$ again we find that 
$$
\begin{aligned}
0&=&2\sum_{a=0}^N \left((\newb{-}\lambda f_a)'\frac{f_a'}{\ps''}+\left(\left(\frac{f_a'}{\ps''}\right)'\right)^2\right)\\
&=&2\sum_{a=0}^N \newb{-}\lambda\frac{(f_a')^2}{\ps''}+\lambda^2 f_a^2,\\
\end{aligned}
$$
and $\sum_{a=0}^N \frac{(f_a')^2}{\ps''}= \lambda \sum_{a=0}^N f_a^2$ and hence it is constant. But, because $\frac{1}{s''}$ vanishes at $0$, $\sum_{a=0}^N \frac{(f_a')^2}{\ps''}=0$ and each $f_a'$ vanishes which is also a contradiction. 
\end{itemize}
\end{proof}
\subsection{Proof of theorem \ref{alld}}
We start with a useful calculation.
\begin{lemma}
In the same context as above, let $f$ be an invariant eigenfunction for the eigenvalue $\lambda$ of the Laplacian on toric K\"ahler manifold with symplectic potential $\ps$ then
\begin{equation}\label{expande_dlambda}
\frac{\del^2 \left(\ps^{il}f_{l} \ps^{jr}f_{r}\right)}{\del x_i \del x_j}=\lambda^2  f^2+2\lambda\del f^t(\hess \ps)^{-1}\del f+\tr(D((\hess \ps)^{-1}\del f))^2
\end{equation}
\end{lemma}
\begin{proof}
\begin{IEEEeqnarray*}{l}
\frac{\del^2 \left(\ps^{il}f_{l} \ps^{jr}f_{r}\right)}{\del x_i \del x_j}=\frac{\del \left(\ps^{il}f_{l}\right)}{\del x_i}\frac{\del\left( \ps^{jr}f_{r}\right)}{\del x_j}+2\frac{\del^2 \left(\ps^{il}f_{l} \right)}{\del x_i \del x_j}\ps^{jr}f_{r}+
\frac{\del (\ps^{il}f_{l})}{\del x_j}\frac{\del\left( \ps^{jr}f_{r}\right)}{\del x_i} \\
=(\lambda f)( \lambda f)+2\frac{\del (\newb{-} \lambda f)}{ \del x_j}\ps^{jr}f_{r}+\frac{\del \left(\ps^{il}f_{l}\right)}{\del x_j}\frac{\del\left( \ps^{jr}f_{r}\right)}{\del x_i} \\
=\lambda^2 f^2\newb{-}2\lambda\del f^t(\hess \ps)^{-1}\del f+\frac{\del \left(\ps^{il}f_{l}\right)}{\del x_j}\frac{\del\left( \ps^{jr}f_{r}\right)}{\del x_i}. \\
\end{IEEEeqnarray*}
Where we have used the fact that
$$
\frac{\del (\ps^{il}f_{l})}{\del x_j}=\newb{-}\lambda f.
$$
Now $$\frac{\del \left(\ps^{il}f_{l}\right)}{\del x_j}=\left[D\left((\hess \ps)^{-1}\del f\right)\right]_{ij}$$ and the result follows.
\end{proof}
As a result of this calculation and of Proposition \ref{characterisation_crit} it follows that
the symplectic potential $\ps$ is $\lT$-critical iff for all $\delta \ps\in \mC^\infty(\bar{\pol})$ there are functions on $\pol$ $\{f_0,\cdots, f_N\}$ corresponding to invariant eigenfunctions in $E_1^\bT(\ps)$  satisfying
$$
\sum_{a=1}^N  \left(\lambda^2f_a^2\newb{-}2\lambda\partial f_a(\hess \ps)^{-1}\del{f}_a+\tr (D(\hess \ps)^{-1}\del{f}_a)^2)\right)=0.
$$

We are now ready to prove our main theorem.
\begin{proof}
Suppose that there exists a $\lT$-critical metric on a toric K\"ahler manifold. We are going to derive a contradiction from this assumption. Let $\pol$ denote the moment polytope of our toric K\"ahler manifold. Assume without loss of generality that $0$ is a vertex of $\pol$ and that $\pol$ is standard at $0$. We can alway achieve this applying an $SL(n,\bZ)$ transformation which will lift to an equivariant diffeomorphism taking critical symplectic potentials for $\lT$ to taking critical symplectic potentials for $\lT$. 

We start by showing that 
\begin{equation}\label{main_crit_relation}
\sum_{a=1}^N  \left(\lambda^2f_a^2\newb{-}2\lambda\partial f_a(\hess \ps)^{-1}\del{f}_a+\tr (D((\hess \ps)^{-1}\del{f}_a)^2)\right)=0
\end{equation}
implies that $f_a(0)=0, \, \forall a=0,\cdots, N.$ The above relation holds at $x=0$. Now $(\hess \ps)^{-1}(0)=0$ and we are going to show that 
$$
\tr (D(\hess \ps)^{-1}\del{f}_a)^2)(0)=\sum_{a=1}^N|\del f_a|^2(0).
$$
It will then follows that $f_a(0),\partial f_a(0) =0, \, \forall a=0,\cdots, N.$ Because $\ps\in \spot(\pol),$ there is $v\in \mC^\infty (\bar{\pol})$ such that  $\ps=\ps_G+v$ and $\ps_G=\sum_{k=1}^d \left(x\cdot \nu_k-c_k\right)\log(x\cdot \nu_k-c_k)-\left(x\cdot \nu_k-c_k\right)$
where
$$
\pol=\left\{x\in \bR^n: x\cdot \nu_l-c_l>0, \, l=1,\cdots d\right\}.
$$
It is not hard to see that 
$$
\hess \ps_G=\sum_{l=1}^d \frac{\nu_l\nu_l^t}{x\cdot \nu_l-c_l}.
$$
Because $\pol$ is standard at zero $\{\nu_1,\cdots,\nu_n\}$ is the canonical basis of $\bR^n$ so that
$$
\hess \ps_G=
\left( \begin{array}{cccc}
\frac{1}{x_1} & 0& \cdots &0  \\
\hfill & \hfill & \ddots &\hfill  \\
0 & \cdots  & 0 & \frac{1}{x_m} \\
\end{array} \right)+A
$$
where $A$ is smooth on a neighbourhood of $0.$ Hence, on a neighbourhood of $0$, there is a smooth $B$ such that
$$
\hess \ps=\left( \begin{array}{cccc}
\frac{1}{x_1} & 0& \cdots &0  \\
\hfill & \hfill & \ddots &\hfill  \\
0 & \cdots  & 0 & \frac{1}{x_n} \\
\end{array} \right)+B.
$$
So 
\begin{equation}\label{hess_0}
(\hess \ps)^{-1}=\mbox{Diag}(x_1,\cdots, x_n)-\mbox{Diag}(x_1,\cdots, x_n)B\mbox{Diag}(x_1,\cdots, x_n)+\cdots
\end{equation}
and therefore writing $\del_lf=f_l,\, l=1,\cdots n$
$$
(\hess \ps)^{-1}\del f=\left( \begin{array}{c}
{x_1}f_1   \\
\vdots   \\
{x_n}f_n \\
\end{array} \right)+O(2)
$$
where for any positive integer $l$, $O(l)$ denotes a function which vanishes to order at least $l$ at zero i.e. a function which is bounded by $c||x||^l$ on some neighbourhood of zero for some constant $C$. Hence 
$$
D((\hess \ps)^{-1}\del f)=\left( \begin{array}{ccc}
\del_1({x_1} f_1)&  \cdots &x_1f_{1n}  \\
\hfill & \ddots &\hfill  \\
x_ nf_{1n}& \cdots& \del_n({x_n} f_n)\\
\end{array} \right)+O(1)
$$
where $f_{ij}=\frac{\del^2f}{\del x_i\del x_j}$ for all $i,j=1,\cdots n$ and 
$$
\tr(D((\hess \ps)^{-1}\del f))^2=\sum_{l=1}^n(f_l+x_lf_{ll})^2+\sum_{l,r=1, l\ne r}^ nx_lx_rf_{lr}^2+O(1).
$$
In particular $\tr(D((\hess \ps)^{-1}\del f))^2(0)=\sum_{l=1}^n(f_l)^2(0)=|\del f|^2(0)$ as claimed. 

Next we want to prove that if we assume that $f_a=O(l)$ for all $a=0,\cdots, N$ and some integer $l>1$ then in fact $f_a=O(l+1).$ Consider the equality 
$$
\sum_{a=0}^N  \left(\lambda^2f_a^2\newb{-}2\lambda\partial f_a(\hess \ps)^{-1}\del{f}_a+\tr (D(\hess \ps)^{-1}\del{f}_a)^2)\right)=0.
$$ 
\begin{itemize}
\item Because $f_a=O(l)$ it follows that $\lambda^2\sum_{a=0}^N  f_a^2=O(2l).$
\item It follows from Equation (\ref{hess_0}) that $(\hess \ps)^{-1}=O(1)$ and since $\del f_a=O(l-1)$,
$\lambda \sum_{a=0}^N \partial f_a(\hess \ps)^{-1}\del{f}_a=O(2l-1).$
\item As for $\sum_{a=0}^N \tr (D(\hess \ps)^{-1}\del{f}_a)^2)$, to study its asymptotic behaviour near $0$ we essentially need to retrace the steps in the above analysis taking into account that $f_a=O(l)$. If $f=O(l)$ then 
$$
(\hess \ps)^{-1}\del f=\left( \begin{array}{c}
{x_1}f_1   \\
\vdots   \\
{x_n}f_n \\
\end{array} \right)+O(l+1)
$$
$$
D((\hess \ps)^{-1}\del f)=\left( \begin{array}{ccc}
\del_1({x_1} f_1)&  \cdots &x_1f_{1n}  \\
\hfill & \ddots &\hfill  \\
x_ nf_{1n}& \cdots& \del_n({x_n} f_n)\\
\end{array} \right)+O(l),
$$
and
$$
\left( \begin{array}{ccc}
\del_1({x_1} f_1)&  \cdots &x_1f_{1n}  \\
\hfill & \ddots &\hfill  \\
x_ nf_{1n}& \cdots& \del_n({x_n} f_n)\\
\end{array} \right)=O(l-1),
$$
so that 
$$
\begin{aligned}
\tr(D((\hess \ps)^{-1}\del f))^2&=&\sum_{l=1}^n(f_l+x_lf_{ll})^2+\sum_{l,r=1, l\ne r}^ nx_lx_rf_{lr}^2+O(2l-1).\\
\end{aligned}
$$
At this point we may conclude that it follows from Equation (\ref{main_crit_relation}) and the analysis above that 
$$
\sum_{a=0}^N\left( (f_{a,l}+x_lf_{a,ll})^2+\sum_{l,r=1}^ nx_lx_rf_{a,lr}^2\right)=O(2l-1),
$$ 
when in fact this expression only needs to be $O(2l-2)$. Consider the analytic expansion of $f_a$ around $0$. We have
$f_a=P_a+O(l+1)$ where $P_a$ is a homogeneous polynomial of order $l$. Therefore
$$
\sum_{a=0}^N\left( (\del_l(x_lP_{a,l}))^2+\sum_{l,r=1}^ nx_lx_rP_{a,lr}^2\right)
$$ 
must be a polynomial of order $2l-1$. Let $v=(x_1,\cdots,x_n)$ be a generic vector in $\{x=(x_1,\cdots,x_n)\in \bR^n: x_1,\cdots,x_n>0\}$ then
$$
t^{2l-2}\sum_{a=0}^N\left( (\del_l(x_lP_{a,l}))^2(v)+\sum_{l,r=1}^ nx_lx_rP_{a,lr}^2(v)\right)
$$
must be of order at least $2l-1$ in $t$ so that 
$$
\sum_{a=0}^N\left( (\del_l(x_lP_{a,l}))^2(v)+\sum_{l,r=1}^ nx_lx_rP_{a,lr}^2(v)\right)=0
$$
and so because all terms in the sum are non negative they must vanish. We conclude that $\del_l(x_lP_{a,l})\equiv0$ and $P_{a,lr}\equiv0$ so that $P_a$ must be constant for all $a=0,\cdots N$. Because $P_a$ is of degree greater than $1$ then it actually must vanish so that $f_a=O(l+1)$ as claimed.
\end{itemize}
Since we have proved that $f_a=O(1)$ and $f_a=O(k)\implies f_a=O(k+1)$ it follows that all derivatives of $f_a$ vanish at zero for all $a=0,\cdots, N$. At this point we use the analyticity hypothesis. Because our Riemannian metric is analytic, the eigenfunctions for its Laplace operator are analytic as well. This follows from elliptic regularity. We may then conclude that all $f_a\equiv 0$ which is impossible. No critical metric exists.
\end{proof}

\section{Concluding remarks}
We would like to be able to use the equations that we derived from the $\lT$-criticality on the metric and the corresponding eigenfunctions to conclude that both metric and eigenfunctions are analytic. The symplectic potential of a $\lT$-critical metric and its eigenfunction satisfy the following system of PDE's for function on $\pol$
\begin{equation}\label{systemPDE}
\begin{cases}
\frac{\del}{\del x_i}\left(\ps^{ij}\frac{\del f_a}{\del x_j}\right)=\lT f_a,\, \forall a=0,\cdots, N\\
\sum_{a=0}^N  \frac{\del^2\left(\ps^{il}f_{a,l}\ps^{jr}f_{a,r}\right)}{\del x_i\del x_j}=0.
\end{cases}
\end{equation}
This can be written in the form $F(x,\ps,f, \del \ps, \del f,\cdots)=0$ for an analytic function $F$ (here we write $f=(f_0,\cdots,f_N)$). It would follow from a result of Morrey (see \cite{m}) that if this system is elliptic in some suitable sense then its solutions are analytic. In fact the system is not elliptic. We will prove this here for the sake of completeness.
\begin{lemma}
The system (\ref{systemPDE}) is nowhere elliptic.
\end{lemma}
\begin{proof}
This is essentially a matter of chasing through the definition of ellipticity. See \cite{m} for more details. Writing $F=(F_0,\cdots, F_N, F_{N+1})$ with
$$
\begin{cases}
F_a=\frac{\del}{x_i}\left(\ps^{ij}\frac{\del f_a}{\del x_j}\right)-\lT f_a,\, \forall a=0,\cdots, N\\
F_{N+1}=\sum_{a=0}^N  \frac{\del^2\left(\ps^{il}f_{a,l}\ps^{jr}f_{a,r}\right)}{\del x_i\del x_j},
\end{cases}
$$
we essentially want to calculate $\det DF$. We start by calculating each partial derivative. We set $f_{N+1}=\ps$ and below we will omit the dependence of $F$ on variables that are fixed.
\begin{enumerate}
\item Given $a=0,\cdots, N$ 
$$
\frac{d }{dt}_{|t=0}F_a(f_a+tv)=\frac{\del}{x_i}\left(\ps^{ij}\frac{\del v}{\del x_j}\right)-\lT v,
$$
so that 
$$
L_{aa}(x,D)=D_i\ps^{ij}D_j=D^t(\hess \ps)^{-1}D,
$$
where we have used the notation in \cite{m}. 
\item Also given $a,b<N+1$ distinct
$$
\frac{d }{dt}_{|t=0}F_a(f_b+tv)=0, \, a\ne b,
$$
so that 
$$
L_{ab}(x,D)=0, \, a\ne b.
$$
\item Now given $a<N+1$ the derivative of $F_a$ with respect to $\ps$ is given by
$$
\frac{d }{dt}_{|t=0}F_a(\ps+tv)=-\frac{\del}{\del x_i}\left(\ps^{il}v_{lr}\ps^{rj}\frac{\del f_a}{\del x_j}\right),
$$
and 
$$
L_{aN+1}(x,D)=-D_i\ps^{il}D_lD_r\ps^{rj}\frac{\del f_a}{\del x_j}=-D^t(\hess \ps)^{-1}DD^t(\hess \ps)^{-1}\del f_a.
$$
\item As for the derivative of $F_{N+1}$  with respect to $f_a$ for $a<N+1$
$$
\frac{d }{dt}_{|t=0}F_{N+1}(f_a+tv)=  2\frac{\del^2 \left(\ps^{il}f_{a,l}\ps^{jr}v_{r}\right)
}{\del x_i\del x_j},
$$
and
$$
\begin{aligned}
L_{N+1,a}(x,D)&=&  2D_iD_j\ps^{il}f_{a,l}\ps^{jr}D_r\\
&=&2D^t(\hess \ps)^{-1}DD^t(\hess \ps)^{-1}\del f_a.
\end{aligned}
$$
\item Last, we calculate the derivative of $F_{N+1}$ with respect to $\ps$
$$
\frac{d }{dt}_{|t=0}F_{N+1}(\ps+tv)=-2\sum_{a=1}^N  \frac{\del^2\left(\ps^{iq}v_{qp}\ps^{pl}f_{a,l}\ps^{jr}f_{a,r}\right)}{\del x_i\del x_j},
$$
and 
$$
\begin{aligned}
L_{N+1N+1}(x,D)&=&-2\sum_{a=0}^N D_iD_j\ps^{iq}D_qD_p\ps^{pl}f_{a,l}\ps^{jr}f_{a,r}\\
&=&-2D^t(\hess\ps)^{-1}D \sum_{a=0}^N(D^t(\hess\ps)^{-1}\del f_a)^2.
\end{aligned}
$$
\end{enumerate}
To sum up
\begin{center}
    \begin{tabular}{|l|l|l| }
    \hline
    &$\frac{d }{dt}_{|t=0}F_a(f_b+tv)$ & $L_{ab}(x,D)$  \\ \hline
    \hline
    $a=b<N+1$&$\frac{\del}{x_i}\left(\ps^{ij}\frac{\del v}{\del x_j}\right)-\lT v$ & $D^t(\hess \ps)^{-1}D$  \\ \hline
    $a,b<N+1, a\ne b$&$0$ & $0$  \\ \hline
    $a<N+1, b=N+1$ &$-2\sum_{a=1}^N  \frac{\del^2 \left(\ps^{il}f_{a,l}\ps^{jr}v_r\right)}{\del x_i\del x_j}$
 & $-2D^t(\hess \ps)^{-1}DD^t(\hess \ps)^{-1}\del f_a$ \\ \hline
    $a=b=N+1$&$-2\sum_{a=1}^N  \frac{\del^2\left(\ps^{iq}v_{qp}\ps^{pl}f_{a,l}\ps^{jr}f_{a,r}\right)}{\del x_i\del x_j}$ & $-2D^t(\hess\ps)^{-1}D \sum_{a=0}^N(D^t(\hess\ps)^{-1}\del f_a)^2$ \\
    \hline
    \end{tabular}
\end{center}
The system is elliptic iff
$$
\det DF:=\det(L_{ij}(x,D))_{i,j=0}^{N+1}\ne 0,\, \forall D\ne 0
$$
Now $DF$ is given by $D^t (\hess \ps)^{-1}D$ times
$$
\left( \begin{array}{cccc}
1&0&  \cdots & -D^t(\hess \ps)^{-1}\del f_0\\
\hfill & \ddots &\hfill &\hfill \\
0& \cdots& 1&-D^t(\hess \ps)^{-1}\del f_N\\
2D^t(\hess \ps)^{-1}\del f_0&\cdots&2D^t(\hess \ps)^{-1}\del f_N&-2\sum_{a=0}^N(D^t(\hess\ps)^{-1}\del f_a)^2\\
\end{array} \right)
$$
The matrix above is clearly singular at all points as its last line is a linear combination of the previous $N$ lines. 
\end{proof}


\begin{thebibliography}{99}
  \bibitem[A]{a}
  {\scshape M. Abreu}
   \emph {K\"ahler geometry of toric manifolds in symplectic coordinates}, in ``Symplectic and Contact Topology: Interactions and Perspectives" (eds. Y.Eliashberg, B.Khesin and F.Lalonde), 
Fields Institute Communications {\bf 35}, American Mathematical Society (2003), 1--24.
    \bibitem[AJK]{ajk}
  {\scshape V. Apostolov, D. Jakobson, G. Konkarev}
   \emph{An extremal eigenvalue problem in K\"ahler geometry}, Conformal and Complex Geometry in Honor of Paul Gauduchon, J. Geo. Phys. {\bf 91} (2015) 108--116.
   \bibitem[BLY]{bly}
  {\scshape J.P. Bourguignon, P. Li, S.T. Yau}
    \emph{Upper bound for the first eigenvalue of algebraic submanifolds}, Comment. Math. Helvetici  {\bf 69} (1994) 199--207.
  \bibitem[CDE]{cde}
   {\scshape B. Colbois, E. Dryden, A. El Soufi}
   \emph{ Extremal G-invariant eigenvalues of the Laplacian of G-invariant metrics}, Math. Z. {\bf 258} (2008), no. 1, 29--41.
   \bibitem[AI]{ei}
{\scshape A.~El~Soufi, S.~Ilias}
\emph{ Riemannian manifolds admitting isometric immersions by their first eigenfunctions}, Pacific J. Math. {\bf 195} (2000), no. 1, 91--99.
 \bibitem[G]{g}
  {\scshape V. Guillemin}
    \emph{K\"ahler structures on toric varieties}, J. Diff. Geom. {\bf 40} (1994), 285--309.
    \bibitem[LS]{ls}
{\scshape E. Legendre, R. Sena-Dias}
\emph{Toric aspects of the first eigenvalue}, to be published in J. Geom. Anal (2017), https://doi.org/10.1007/s12220-017-9908-y.
\bibitem[Ma]{mb}
{\scshape H. Macbeth}
\emph{Conformal classes realizing the Yamabe invariant}, to be published in Int. Math. Res. Not. IMRN, https://doi.org/10.1093/imrn/rnx123.

\bibitem[Mo]{m}
{\scshape C. Morrey}
 \emph{On the analyticity of the solutions of analytic non-linear elliptic systems of partial differential equations. Part I. Analyticity in the interior}, Amer. J. Math. {\bf 80} (1958) 198--218.
 \bibitem[N]{n}
{\scshape N. Nadirashvili} 
\emph{ Berger's isoperimetric problem and minimal immersions of surfaces}, Geom. Funct. Anal. {\bf 6} (1996), no. 5, 877--897.
\end{thebibliography}
\end{document}